\documentclass{amsart}
\usepackage{amsmath,amssymb,mathrsfs,amscd}
\usepackage[left=3.5cm,right=3.5cm]{geometry}

\usepackage{color}
\usepackage{multirow}
\usepackage{enumerate}
\usepackage{xypic}

\usepackage{dsfont}    
\usepackage{amsthm}    
\usepackage{graphicx}  

\usepackage{setspace}

\theoremstyle{plain}
\newtheorem{thm}{Theorem}[section]
\newtheorem{remark}[thm]{Remark}

\newtheorem{pro}[thm]{Proposition}
\newtheorem{eg}[thm]{Example}

\newtheorem{cor}[thm]{Corollary}

\def\P{\mathds{P}}
\def\O{\mathcal{O}}
\def\E{\mathcal{E}}
\def\F{\mathcal{F}}
\def\G{\mathcal{G}}
\def\Z{\mathbb{Z}}

\def\dim{\mathrm{dim}}

\begin{document}

\title{Ulrich Bundles on Projective Spaces}
\author{Zhiming Lin}

\address{Peking University\\Beijing\\P. R. China}
\email{zmlin@pku.edu.cn}

\subjclass[2010]{Primary 14J60}

\keywords{Ulrich bundles, projective spaces, Veronese embedding.}


\maketitle


\begin{abstract}
We assume that $\E$ is a rank $r$ Ulrich bundle for $(\P^n, \O(d))$. The main result of this paper is that $\E(i)\otimes \Omega^{j}(j)$ has natural cohomology for any integers $i \in \Z$ and $0 \leq j \leq n$, and every Ulrich bundle $\E$ has a resolution in terms of $n$ of the trivial bundle over $\P^n$. As a corollary, we can give a necessary and sufficient condition for Ulrich bundles if $n \leq 3$, which can be used to find some new examples, i.e., rank $2$ bundles for $(\P^3, \O(2))$ and rank $3$ bundles for $(\P^2, \O(3))$.
\end{abstract}

\allowdisplaybreaks

\section{Introduction}

Let $\E$ be a vector bundle and $H$ be a very ample line bundle on a variety $X$. We say $\E$ is an Ulrich bundle for $(X,H)$, if $h^{q}(X,\E(-pH))=0$ for each $q \in \Z$ and $1\leq p \leq \dim(X)$. This notion was introduced in \cite{Eisenbud-Schreyer-Weyman-03}, where various other characterizations are given, i.e., an Ulrich bundle implies a linear maximal Cohen-Macaulay module. Ulrich asked in \cite{Ulrich-84} whether every local Cohen-Macaulay ring admits a linear MCM-module. Correspondingly, it was conjectured in \cite{Eisenbud-Schreyer-Weyman-03} that whether there exists  an Ulrich bundle on any variety for any embedding, and what is the smallest possible rank for such Ulrich bundles.

This conjecture is a wide open problem, and we know a few scattered results at present: \cite{Herzog-Ulrich-Backelin-91} for hypersurfaces and complete intersections, \cite{Eisenbud-Schreyer-Weyman-03} for curves and del Pezzo surfaces, \cite{Beauville-15.12} for abelian surfaces, \cite{Aprodu-Costa-Miro Roig-16} for geometrically ruled surfaces of invariant $e>0$, \cite{Aprodu-Farkas-Ortega-12} for sufficiently general K3 surfaces, \cite{Beauville-16.10} for bielliptic surfaces, \cite{Beauville-16.07} and \cite{Casnati-16.09} for non-special surfaces with $p_g=q=0$ (i.e. Enriques surfaces).

In this paper, we will work on an algebraically closed field $k$ of characteristic $0$ and $\P^n$ will denote the projective space over $k$ of dimension $n$. We consider the Ulrich bundles on $\P^n$ for the Veronese embedding $\O(d)$.

In \cite{Beauville-16.10}, Beauville proved that for any positive integers $n$ and $d$, there is an Ulrich bundle $\E$ of rank $n!$ for $(\P^n, \O(d))$, using the quotient map $\pi : (\P^1)^{n} \rightarrow \mathrm{Sym}^{n}\P^1 =\P^n$ of degree $n!$. So the existence of Ulrich bundles for $(\P^n, \O(d))$ has been solved completely. But as $n$ increases, $n!$ grows very fast. So we want to know what is the smallest possible rank for such bundles. In Section 3, we will prove that there is a rank $1$ Ulrich bundle $\E$ for $(\P^n,\O(d))$ if and only if $d=1, \E \simeq \O_{\P^n}$ or $n=1, \E\simeq \O(d-1)$.

In \cite{Eisenbud-Schreyer-Weyman-03}, Eisenbud, Schreyer and Weyman proved that if $\E$ is an Ulrich bundle for $(\P^n, \O(d))$, then $\E(i)$ has natural cohomology for any integer $i$. In Section 4, we will generalize this and prove the following result:

\bigskip
\noindent \textbf{Theorem \ref{pro-space-Lin-cohomology}.}\;\;If $\E$ be an Ulrich bundle for $(\P^n,\O(d))$ and $1\leq j \leq n$, then
\begin{equation*}
  h^q(\E(i)\otimes \Omega^j(j)) \neq 0 \Rightarrow
  \begin{cases}
    q=0\text{ and }-d < i, &  \\
    0 < q < n\text{ and }-(q + 1)d < i \leq -qd, &  \\
    q=n\text{ and }i \leq -nd. &
  \end{cases}
\end{equation*}
In particular, $\E(i)\otimes \Omega^j(j)$ has natural cohomology. Thus all the $h^q(\E(i)\otimes \Omega^j(j))$ are determined by the formula
$$\chi (\E(i)\otimes \Omega^j(j)) = rd^n\sum_{k=0}^{j}(-1)^{j-k}\binom{n+1}{k}\binom{\frac{i+j-k}{d}+n}{n}.$$

\bigskip
Based on this, we prove that every Ulrich bundle has a resolution in terms of $n$ of the trivial bundle, which is the generalization of Corollary 4.3 in \cite{Coskun-Genc-16}:

\bigskip
\noindent \textbf{Theorem \ref{pro-space-Lin-sequence}.}\;\;If $\E$ is an Ulrich bundle of rank $r$ for $(\P^n, \O(d))$, then we have an exact sequence
\begin{equation*}
    0 \rightarrow \O^{a_n}(-n)\rightarrow \cdots \rightarrow \O^{a_2}(-2)\rightarrow \O^{a_1}(-1)\rightarrow \E(-d) \rightarrow 0,
\end{equation*}
where
$$a_j=-rd^n\sum_{k=0}^{j}(-1)^{j-k}\binom{n+1}{k}\binom{\frac{j-k}{d}+n-1}{n}.$$

\bigskip
As a corollary, we can give a necessary and sufficient condition for Ulrich bundles if $n \leq 3$, and reclassify all Ulrich bundles for $(\P^n,\O(1))$ and $(\P^1,\O(d))$. In Section 5, we will construct some Ulrich bundles and prove some properties using the exact sequence, i.e., new examples of rank $2$ for $(\P^3, \O(2))$ and of rank $3$ for $(\P^2, \O(3))$.

\section{Preliminaries}

For convenience, we recall two formulas to calculate the cohomology of bundles on $\P^n$.

\begin{thm}
(Bott formula) The values of $h^{q}(\mathds{P}^{n},\Omega^{p}_{\mathds{P}^{n}}(d))$ are given by
\begin{equation*}
  h^{q}(\mathds{P}^{n},\Omega^{p}_{\mathds{P}^{n}}(d))=
  \begin{cases}
    \binom{d+n-p}{d}\binom{d-1}{p} & \mbox{if } q=0,0\leq p\leq n, d>p \\
    1 & \mbox{if } d=0,0\leq p=q\leq n \\
    \binom{-d+p}{-d}\binom{-d-1}{n-p} & \mbox{if } q=n,0\leq p\leq n, d< p-n \\
    0 & \mbox{otherwise}
  \end{cases}
\end{equation*}
\end{thm}
\begin{proof}
See Okonek, Schneider and Spindler's book \cite{Okonek-Schneider-Spindler-88}, Chapter 1 Section 1.1.
\end{proof}

\begin{thm}\label{pro-space-Eisenbud}
Let $\E$ be a rank $r$ vector bundle on $\P^n$. Then $\E$ is an Ulrich bundle for $(\P^n,\O(d))$ if and only if
\begin{equation*}
  h^q(\E(p)) \neq 0 \Leftrightarrow
  \begin{cases}
    q=0\text{ and }-d < p, &  \\
    0 < q < n\text{ and }-(q + 1)d < p < -qd, &  \\
    q=n\text{ and }p < -nd. &
  \end{cases}
\end{equation*}
Thus $\E(p)$ has natural cohomology and all $h^q(\E(p))$ are determined by $\chi (\E(p)) = rd^n \binom{\frac{p}{d}+n}{n}$.
\end{thm}
\begin{proof}
See Eisenbud, Schreyer and Weyman's paper \cite{Eisenbud-Schreyer-Weyman-03}, Theorem 5.1.
\end{proof}

Here, we say a vector bundle $\F$ on $\P^n$ has natural cohomology if for each $p\in \Z$, at most one of the cohomology groups $H^q(\F(p)) \neq 0$ for $0 \leq q \leq n$. For any real number $x$ and positive integers $n$ and $m$, we use the formal operation
\begin{equation*}
  \binom{x+m}{n}\doteq \frac{(x+m)(x+m-1)(x+m-2)\cdots (x+m+1-n)}{n(n-1)(n-2)\cdots 1}.
\end{equation*}

\begin{figure}[htbp]
  \centering
  \includegraphics[width=0.9\textwidth]{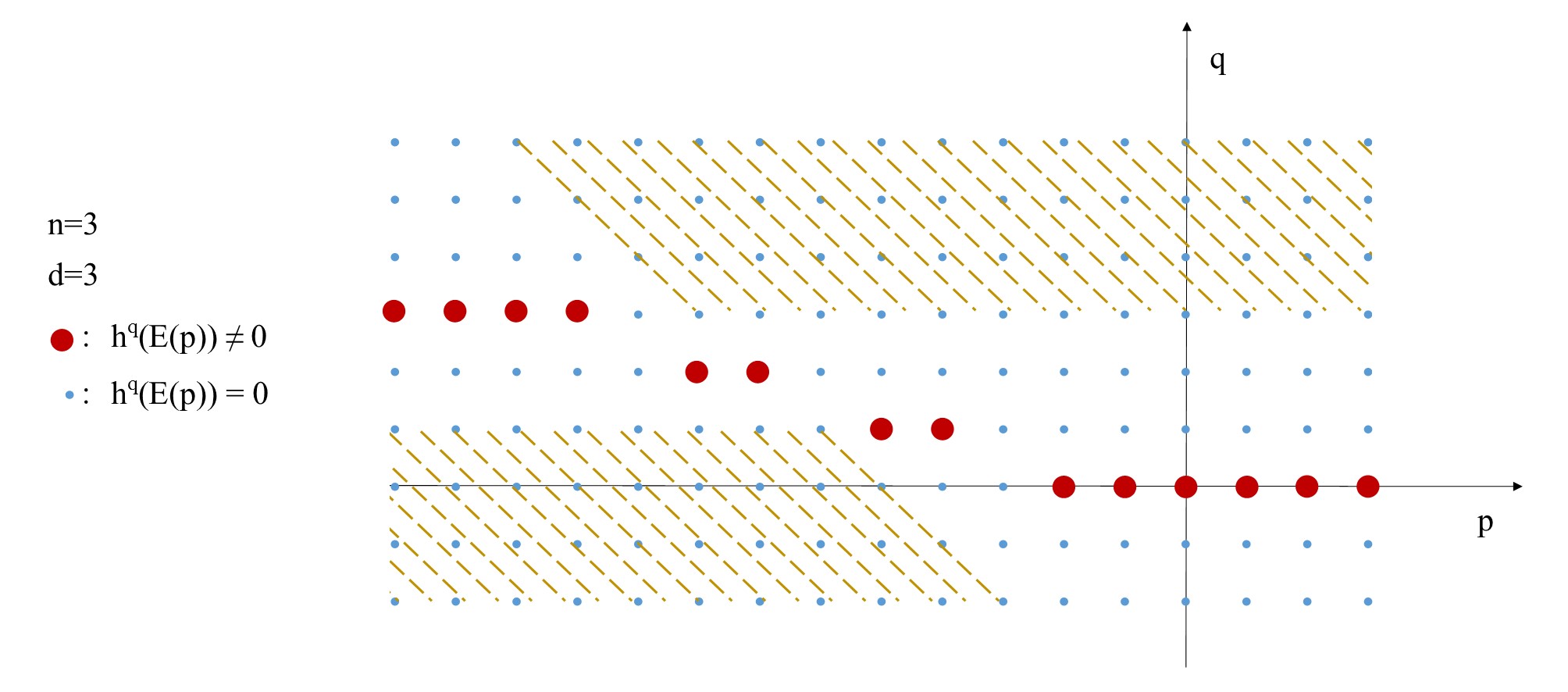}
\end{figure}

For example, if we assume that $\E$ is an Ulrich bundle for $(\P^3,\O(3))$, we have the diagrammatic drawing of the cohomologies as above. Since $h^2(\E(-7))\neq 0$, we can get $h^q(\E(i))=0$ for $q\leq 1, q+i\leq -5$ and $q\geq 3, q+i \geq -5$. In general, we have the following useful tip

\begin{cor}\label{pro-space-Eisenbud-cor-2}
Let $\E$ be an Ulrich bundle for $(\P^n,\O(d))$. If $h^{q_0}(\E(i_0))\neq 0$, then $h^q(\E(i))= 0$ for $q \leq q_0-1, q+i \leq q_0+i_0$ and $q \geq q_0+1, q+i \geq q_0+i_0$.
\end{cor}
\begin{proof}
It's an immediate consequence of Theorem \ref{pro-space-Eisenbud}.
\end{proof}

\section{Basic Properties}

In this section, we construct all rank 1 Ulrich bundles and give a criteria of rank 2 Ulrich bundles.

\begin{pro}\label{pro-space-rank 1}
There is a rank 1 Ulrich bundle $\E$ for $(\P^n,\O(d))$ if and only if $d=1$, $\E \simeq \O_{\P^n}$ or $n=1$, $\E\simeq \O(d-1)$.
\end{pro}
\begin{proof}
Let $\E\doteq \O(a)$ be a rank 1 Ulrich bundle for $(\P^n,\O(d))$. If $n=1$, by definition we have $h^0(\P^1,\E(-d))=h^1(\P^1,\E(-d))=0$, which implies that $\E(-d) \simeq \O(-1)$.  If $n\geq 2$, there are two ways to calculate the Hilbert polynomial
\begin{equation*}
  \begin{cases}
    \chi(\O(a+dt))=d^n \binom{t+n}{n} & \text{by Theorem \ref{pro-space-Eisenbud}} \\
    \chi(\O(a+dt))=\binom{a+dt+n}{n} & \text{by Bott formula}
  \end{cases}
\end{equation*}
Comparing the coefficients of $t^{n-1}$ and $t^0$, we obtain the equations
\begin{equation}
  \begin{cases}
    (n+1)(d-1)=2a &  \\
    d^n=\binom{a+n}{n} &
  \end{cases}
\end{equation}
If we define
\begin{equation*}
    f(d)\doteq ln \binom{a+n}{n}-ln(d^n)
    =\sum_{i=1}^{n}ln[\frac{(n+1)(d-1)}{2}+i]-\sum_{i=1}^{n}lni-nlnd\;\;(d\geq1),
\end{equation*}
then the derivative of $f(d)$ is
\begin{equation*}
    f'(d)=\sum_{i=1}^{n}[\frac{1}{d-1+\frac{2i}{n+1}}-\frac{1}{d}]= \sum_{i=1}^{n}\frac{1}{2}[\frac{1}{d-1+\frac{2i}{n+1}}+\frac{1}{d-1+\frac{2(n+1-i)}{n+1}}-\frac{2}{d}]
    >0.
\end{equation*}
So we have $f(1)=0$ and $f'(d)>0$ for all $d\geq 1$, which implies that $f(d)=0$ if and only if $d=1$. Solving the equations (1), we obtain $d = 1$ and $a = 0$.

The converse is trivial by Bott formula.
\end{proof}

\begin{remark}
Coskun and Genc have proved the case $n=2$ in \cite{Coskun-Genc-16} Proposition 3.1. Besides, in \cite{Beauville-16.10} Beauville claimed that there is a rank 1 Ulrich bundle for $(X,H)$ with $\mathrm{Pic}(X)= \Z \O_X(H)$ if and only if $X \simeq \P^{n}$ for some $n \in \mathds{N}_{+}$, which is different from Proposition \ref{pro-space-rank 1}.
\end{remark}

\begin{pro}\label{pro-space-rank 2}
The rank 2 vector bundle $\E$ is an Ulrich bundle for $(\P^n,\O(d))$, if and only if it satisfies the following two conditions:
\begin{enumerate}[1)]
  \item $c_1(\E)=(n+1)(d-1)$;
  \item If $n=2m$ is even, then $h^q(\E(-pd))=0(\forall q \in \Z,1\leq p \leq m)$; If $n=2m+1$ is odd, then $h^q(\E(-pd))=0(\forall q \in \Z,1\leq p \leq m)$ and $h^q(\E(-(m+1)d))=0(0\leq q \leq m)$.
\end{enumerate}
\end{pro}
\begin{proof}
let $\E$ be a rank 2 Ulrich bundle for $(\P^n,\O(d))$. If we assume that the topological equivalence class of $\E$ is $\O(a)\oplus \O(b)$, then $c_1(\E)=a+b$. There are two ways to calculate the Hilbert polynomial
\begin{equation*}
  \begin{cases}
    \chi(\E(dt))=2d^n \binom{t+n}{n} & \text{by Theorem \ref{pro-space-Eisenbud}} \\
    \chi(\O(a+dt)\oplus \O(b+dt))=\binom{a+dt+n}{n}+\binom{b+dt+n}{n} & \text{by Bott formula}
  \end{cases}
\end{equation*}
Comparing the coefficients of $t^{n-1}$, we obtain $c_1(\E)=(n+1)(d-1)$. Condition 2) is trivial by the definition of Ulrich bundles.

Now we prove the converse. Since $\mathrm{rank}(\E)=2$, we have $\E\simeq \E^{\vee}(\det\E)$, which implies that $\E\simeq \E^{\vee}((n+1)(d-1))$ by condition 1). So we get $h^q(\E(-pd))=h^{n-q}(\E(-(n+1-p)d))$ by Serre duality.
\end{proof}

\begin{remark}
Actually, if we compare the coefficients of $t^{i}$ for all $i$, we can calculate all chern classes of $\E$ of any rank $r$. For example, if $r=2$, we have $c_2(\E)=\frac{(n+1)(d-1)}{12}[(3n+4)d-(3n+2)]$. As a corollary, when $n=3, d=3k$ or $n=4, d=2k$ for any positive integer $k$, there is no Ulrich bundles of rank 2 for $(\P^n,\O(d))$, since $c_2(\E) \in \Z$. Eisenbud, Schreyer and Weyman give a more effective criteria to find the triple $(n,d,r)$ such that there is no Ulrich bundles of rank $r$ for $(\P^n,\O(d))$, see \cite{Eisenbud-Schreyer-Weyman-03} Corollary 5.3.
\end{remark}

\section{Main Results}

For convenience, we denote the sheaf of germs of holomorphic $j$-forms on $\P^n$ by $\Omega^j$. If we dualize the Euler sequence and takes the $j$th exterior power (see Hirzebruch \cite{Hirzebruch-66}, p. 55), we get the following exact sequence of vector bundles
$$ 0\rightarrow \Omega^j(j) \rightarrow \O ^{\oplus\binom{n+1}{j}}\rightarrow \Omega^{j-1}(j)\rightarrow 0.$$
Now we use the notation $\E_{i,j}\doteq \E(i)\otimes \Omega^j(j)$. When we tensor the above sequence by $\E(i)$, we have a series of exact sequences
\begin{equation*}
  (\Lambda_{i,j}):\;\;\;\;0\rightarrow \E_{i,j} \rightarrow \E(i) ^{\oplus\binom{n+1}{j}}\rightarrow \E_{i+1,j-1}\rightarrow 0.
\end{equation*}
These sequences give us a way to calculate the cohomology of $\E(i)\otimes \Omega^j(j)$ as follow.

\begin{thm}\label{pro-space-Lin-cohomology}
If $\E$ be an Ulrich bundle for $(\P^n,\O(d))$ and $1\leq j \leq n$, then
\begin{equation*}
  h^q(\E(i)\otimes \Omega^j(j)) \neq 0 \Rightarrow
  \begin{cases}
    q=0\text{ and }-d < i, &  \\
    0 < q < n\text{ and }-(q + 1)d < i \leq -qd, &  \\
    q=n\text{ and }i \leq -nd. &
  \end{cases}
\end{equation*}
In particular, $\E(i)\otimes \Omega^j(j)$ has natural cohomology. Thus all the $h^q(\E(i)\otimes \Omega^j(j))$ are determined by the formula
$$\chi (\E(i)\otimes \Omega^j(j)) = rd^n\sum_{k=0}^{j}(-1)^{j-k}\binom{n+1}{k}\binom{\frac{i+j-k}{d}+n}{n}.$$
\end{thm}
\begin{proof}
  To prove the first part of this thoerem, we discuss the value of $i$.

  \bigskip

  (Case 1) If $-(q_0+1)d<i< -q_0d$, where $q_0$ is an integer with $1\leq q_0\leq n$.

   Since $h^{q_0}(\E(i))\neq0$ by Theorem \ref{pro-space-Eisenbud}, we discuss the value of $q$. For any $q \leq q_0-1$, we have $h^{q}(\E(i))=h^{q-1}(\E(i))=0$, which implies  $h^{q}(\E_{i,j})=h^{q-1}(\E_{i+1,j-1})$ by the exact sequence $(\Lambda_{i,j})$. Moreover, since $h^{\widehat{q}}(\E(\widehat{i}))=0(\widehat{q}\leq q_0-1, \widehat{q}+\widehat{i}\leq q_0+i)$ by Corollary \ref{pro-space-Eisenbud-cor-2}, we can repeat the above step and conclude that
   $$h^{q}(\E_{i,j})=h^{q-1}(\E_{i+1,j-1})=\cdots=h^{q-j}(\E_{i+j,0})=h^{q-j}(\E(i+j))=0.$$
   For any $q \geq  q_0+1$, we have $h^{q}(\E(i-1))=h^{q+1}(\E(i-1))=0$ by Corollary \ref{pro-space-Eisenbud-cor-2}, which implies $h^{q}(\E_{i,j})=h^{q+1}(\E_{i-1,j+1})$ by the exact sequence $(\Lambda_{i-1,j+1})$. Moreover, with the same reason, we have $h^{\widehat{q}}(\E(\widehat{i}))=0(\widehat{q}\geq q_0+1, \widehat{q}+\widehat{i}\geq q_0+i)$, so we can also repeat the above step and conclude that
   $$h^{q}(\E_{i,j})=\cdots=h^{q+(n-j)}(\E_{i-(n-j),n})=h^{q+(n-j)}(\E(i-(n-j)-1))=0.$$
   Hence, at this case we have $h^{q}(\E_{i,j})=0$ for $q\neq q_0$.

  \bigskip

  (Case 2) If $i >-d$.

  As in case 1, since $h^{0}(\E(i))\neq0$, we have $h^{\widehat{q}}(\E(\widehat{i}))=0(\widehat{q}\geq 1, \widehat{q}+\widehat{i}\geq i)$. Then for any $q \geq  1$, we can get $h^{q}(\E_{i,j})=h^{q+(n-j)}(\E_{i-(n-j),n})=h^{q+(n-j)}(\E(i-(n-j)-1))=0$. Hence, at this case we have $h^{q}(\E_{i,j})=0$ for $q\neq 0$.

  \bigskip

  (Case 3) If $i <-nd$.

  As in case 1, since $h^{n}(\E(i))\neq0$, we have $h^{\widehat{q}}(\E(\widehat{i}))=0(\widehat{q}\leq n-1, \widehat{q}+\widehat{i} \leq n+i)$. Then for any $q \leq  n-1$, we can get $h^{q}(\E_{i,j})=h^{q-j}(\E_{i+j,0})=h^{q-j}(\E(i+j))=0$. Hence, at this case we have $h^{q}(\E_{i,j})=0$ for $q\neq n$.

  \bigskip

  (Case 4) If $i =-q_0d$, where $q_0$ is an integer with $1\leq q_0\leq n$.

  Since $h^{q}(\E(i))=0(\forall q)$, we have $h^{q}(\E_{i,j})=h^{q-1}(\E_{i+1,j-1})(\forall q)$ by the exact sequence $(\Lambda_{i,j})$. Now we only need to calculate $h^q(\E_{i+1,j-1})(\forall q)$. If $d\neq 1$, then $-q_0d< i+1 <-(q_0-1)d$. It turns into case 1 or case 2, which we have done. If $d=1$, then $i =-q_0$. For $j \leq q_0-1$, we have $-n\leq i+j\leq -1$ and $h^{q}(\E_{i,j})=h^{q+j}(\E_{i+j,0})=h^{q+j}(\E(i+j))=0$ for any $q$. For $j \geq q_0+1$, we have $-n \leq i+j-n-1 \leq -1$ and $h^{q}(\E_{i,j})=h^{q-(n-j)}(\E_{i-(n-j),n})=h^{q-(n-j)}(\E(i+j-n-1))=0$ for any $q$. If $j = q_0$, using the same method we have $h^{q}(\E_{-q_0,q_0})\neq0$ if and only if $q=q_0$. Hence, at all of these cases we have $h^{q}(\E_{i,j})=0$ for $q\neq q_0$.

  \bigskip

  Now we prove the last part of the theorem by induction. By the exact sequence $(\Lambda_{i,j})$, we have
  $$\chi(\E_{i,j})=\binom{n+1}{j}\chi(\E(i))-\chi(\E_{i+1,j-1}).$$
  By Theorem \ref{pro-space-Eisenbud} and induction, we can get
  $$\chi(\E_{i,j})=\sum_{k=0}^{j}(-1)^{j-k}\binom{n+1}{k}\chi(\E(i+j-k))= rd^n\sum_{k=0}^{j}(-1)^{j-k}\binom{n+1}{k}\binom{\frac{i+j-k}{d}+n}{n}.$$
  So we have done.
\end{proof}

\begin{remark}\label{pro-space-Lin-cohomology-remark}
If we want to find a sufficient and necessary condition as in Theorem \ref{pro-space-Eisenbud}, we need to determinate all quintuple $(n,d,r,i,j)$ such that $\chi(\E_{i,j})= 0$, which is very difficult. But for $d=1$, we have proved that $\chi(\E_{i,j})=0$ if $-n\leq i \leq -1$ and $i+j\neq 0$.
\end{remark}

\begin{thm}\label{pro-space-Lin-sequence}
  If $\E$ is an Ulrich bundle of rank $r$ for $(\P^n, \O(d))$, then we have an exact sequence
  \begin{equation}\label{pro-space-Lin-sequence-eq}
    0 \rightarrow \O^{a_n}(-n)\rightarrow \cdots \rightarrow \O^{a_2}(-2)\rightarrow \O^{a_1}(-1)\rightarrow \E(-d) \rightarrow 0,
  \end{equation}
  where
  $$a_j=-rd^n\sum_{k=0}^{j}(-1)^{j-k}\binom{n+1}{k}\binom{\frac{j-k}{d}+n-1}{n}.$$
\end{thm}
\begin{proof}
  By Beilinson's theorem (\cite{Okonek-Schneider-Spindler-88}, Chapter 2 Theorem 3.1.4), for any vector bundle $\F$ on $\P^n$ there is a spectral sequence
  \begin{equation*}
    E_1^{p,q} = H^q(\P^n, \F \otimes \Omega_{\P^n}^{-p}(-p))\otimes \O_{\P^n}(p) \Rightarrow E^{p+q}=
    \begin{cases}
      \F, & \mbox{if } p+q=0, \\
      0, & \mbox{otherwise}.
    \end{cases}
  \end{equation*}
  Let $\F=\E(-d)$. By Theorem \ref{pro-space-Lin-cohomology}, we have $h^q(\F \otimes \Omega^{-p}(-p))=0$ for $q\neq 1$ and $h^1(\F)=0$. By the properties of spectral sequence, we have $E_{2}^{p,1}=E_{\infty}^{p,1}=0 $ for $p\neq -1$ and $E_{2}^{-1,1}=E_{\infty}^{-1,1}=\F $, which gives the exact sequence (\ref{pro-space-Lin-sequence-eq}). At this time, we have
  $$a_j=-\chi(\E(-d)\otimes \Omega^{j}(j))=-rd^n\sum_{k=0}^{j}(-1)^{j-k}\binom{n+1}{k}\binom{\frac{-d+j-k}{d}+n}{n}.$$
  So we have done.
\end{proof}

\begin{cor}\label{pro-space-Lin-sequence-cor}
  A rank $r$ vector bundle $\E$ is an Ulrich bundle for $(\P^n,\O(1))$ if and only if $\E\simeq \O^r$.
\end{cor}
\begin{proof}
  If $\E\simeq \O^r$, then $h^q(\E(-p))=0$ for any $q \in \Z$ and $1 \leq p \leq n$ by Bott formula. If $\E$ is an Ulrich bundles for $(\P^n,\O(1))$, we have an exact sequence $0\rightarrow \O^r(-1)\rightarrow \E(-1) \rightarrow 0$ by Theorem \ref{pro-space-Lin-sequence} and Remark \ref{pro-space-Lin-cohomology-remark}.
\end{proof}

\begin{cor}\label{pro-space-Lin-sequence-cor1}
  A rank $r$ vector bundle $\E$ is an Ulrich bundle for $(\P^1,\O(d))$ if and only if $\E\simeq \O^r(d-1)$.
\end{cor}
\begin{proof}
  If $\E\simeq \O^r(d-1)$, then $h^0(\E(-d))=h^1(\E(-d))=0$ by Bott formula. If $\E$ is an Ulrich bundles for $(\P^1,\O(d))$, we have an exact sequence $0\rightarrow \O^r(-1)\rightarrow \E(-d) \rightarrow 0$ by Theorem \ref{pro-space-Lin-sequence}.
\end{proof}

\begin{cor}\label{pro-space-Lin-sequence-cor2}
  A rank $r$ vector bundle $\E$ is an Ulrich bundle for $(\P^2,\O(d))$ if and only if $\E$ satisfies the following two conditions:
  \begin{enumerate}[1)]
    \item There is an exact sequence
      $$0 \rightarrow \O^{\frac{r}{2}(d-1)}(-2)\rightarrow \O^{\frac{r}{2}(d+1)}(-1)\rightarrow \E(-d) \rightarrow 0;$$
    \item The induced map $H^2(\alpha(-d)):H^2(\O^{\frac{r}{2}(d-1)}(-d-2)) \rightarrow H^2(\O^{\frac{r}{2}(d+1)}(-d-1))$ is injective (or surjective).
  \end{enumerate}
  In particular, if $r=2$, the condition 2) is trivial.
\end{cor}
\begin{proof}
  If $\E$ is an Ulrich bundles for $(\P^2,\O(d))$, we have the condition 1) by Theorem \ref{pro-space-Lin-sequence}. The exact sequence implies $h^q(\E(-d))=0$ for all $q \in \Z$ and $\chi(\E(-2d))=h^0(\E(-2d))=0$. Condition 2) is equivalent to $h^1(\E(-2d))=0$ (or $h^2(\E(-2d))=0$).

  If $r=2$, the exact sequence implies that $c_1(\E)=3(d-1)$. By Proposition \ref{pro-space-rank 2}, the condition 2) is trivial.
\end{proof}

\begin{cor}\label{pro-space-Lin-sequence-cor3}
  A rank $r$ vector bundle $\E$ is an Ulrich bundle for $(\P^3,\O(d))$ if and only if $\E$ satisfies the following two conditions:
  \begin{enumerate}[1)]
    \item There is an exact sequence
        $$0 \rightarrow \O^{\frac{r}{6}(2d-1)(d-1)}(-3)\rightarrow \O^{\frac{2r}{3}(d^2-1)}(-2)\rightarrow \O^{\frac{r}{6}(2d+1)(d+1)}(-1)\rightarrow \E(-d) \rightarrow 0;$$
    \item The induced map $H^3(\alpha(-id)):H^3(\O^{\frac{r}{6}(2d-1)(d-1)}(-id-3)) \rightarrow H^3(\O^{\frac{2r}{3}(d^2-1)}(-id-2))$ is injective and $H^3(\beta(-id)):H^3(\O^{\frac{2r}{3}(d^2-1)}(-id-2)) \rightarrow H^3(\O^{\frac{r}{6}(2d+1)(d+1)}(-id-1))$ is surjective for $1\leq i\leq2$.
  \end{enumerate}
  In particular, if $r=2$, the condition 2) is replaced by that $H^2(\alpha(-d))$ is injective.
\end{cor}
\begin{proof}
  Let $\F$ be the cokernel of $\alpha$. The exact sequence in condition 1) is equivalent to two short exact sequence
  $$0 \rightarrow \O^{\frac{r}{6}(2d-1)(d-1)}(-3) \rightarrow \O^{\frac{2r}{3}(d^2-1)}(-2) \rightarrow \F \rightarrow 0,$$
  and
  $$0\rightarrow \F \rightarrow \O^{\frac{r}{6}(2d+1)(d+1)}(-1) \rightarrow \E(-d) \rightarrow 0.$$
  So the sequences imply that $\chi(\E(-id))=h^0(\E(-id))=h^q(\E(-d))=0$ for $2\leq i\leq3$ and all $q \in \Z$. The injection of $H^3(\alpha(-id))$ is equivalent to $h^2(\F(-id))=0$, which is equivalent to $h^1(\E(-(i+1)d))=0$. The surjection of $H^3(\beta(-id))$ is equivalent to $h^3(\E(-id))=0$.

  If $r=2$, the exact sequence implies that $c_1(\E)=4(d-1)$. By Proposition \ref{pro-space-rank 2}, to prove that the cokernel $\E$ in condition 1) is an Ulrich bundle, we only need to prove $h^1(\E(-2d))=0$, which is is equivalent to the injection of $H^3(\alpha(-d))$.
\end{proof}

\begin{remark}
For $n=2$, Coskun and Genc have found this exact sequence in \cite{Coskun-Genc-16}. The results in Corollary \ref{pro-space-Lin-sequence-cor} and Corollary \ref{pro-space-Lin-sequence-cor1} are well known, and there are other proofs using Grothendieck’s theorem and the splitting criterion of Horrocks.
\end{remark}

\section{Applications}

By Proposition \ref{pro-space-rank 1}, Corollary \ref{pro-space-Lin-sequence-cor} and Corollary \ref{pro-space-Lin-sequence-cor1}, if we want to find a non-trivial Ulrich bundle $\E$ of rank $r$ for $(\P^n,\O(d))$, then we must have $r\geq 2, d\geq 2$ and $n\geq 2$. In this section, we will give some interesting examples for $r=2$ or $3$.

\begin{eg}
  Consider the case $n=2$ and $r=2$. By Corollary \ref{pro-space-Lin-sequence-cor2}, finding a rank 2 Ulrich bundle for $(\P^2, \O(d))$ is equivalent to finding a morphism $\alpha : \O^{d-1}(-2) \rightarrow \O^{d+1}(-1)$ such that the cekernel of $\alpha$ is locally free. For example, if we assume that $\{x,y,z\}$ are the homogeneous coordinates of $\P^2$, we can take
  \begin{equation*}
    \alpha=\left(
    \begin{array}{ccccccc}
      x & y & z & 0 & 0 &\cdots & 0 \\
      0 & x & y & z & 0 &\cdots & 0 \\
      0 & 0 & x & y & z &\cdots & 0 \\
      \vdots & \vdots & \vdots &\vdots & \vdots & \ddots & \vdots \\
      0 & \cdots & 0 & 0&x & y & z
    \end{array}
    \right).
  \end{equation*}
  For any point $p \in \P^2$ we have $\mathrm{rank}(\alpha_{p})=d-1$, then the cekernel $\G$ of $\alpha$ must be locally free. Hence, $\G(d)$ is an Ulrich bundle of rank $2$ for $(\P^2, \O(d))$.
\end{eg}

This example is given by Eisenbud, Schreyer and Weyman in \cite{Eisenbud-Schreyer-Weyman-03}  Proposition 5.9. There are two other proofs for the existence of rank 2 Ulrich bundles for $(\P^2, \O(d))$ given by Beauville and Casnati, in \cite{Beauville-16.10} Proposition 4 and \cite{Casnati-16.09} Theorem 1.1. On one hand, by Beauville's result, there exists an Ulrich bundle of rank $2!$ for any $(\mathds{P}^{2}, \mathcal{O}(d))$. On another hand, by Serre construction, there exists a rank 2 Ulrich bundle $\E$ for $(\mathds{P}^{2}, \mathcal{O}(d))$ fitting into the exact sequence
$$0 \rightarrow \O(d-3) \rightarrow \E \rightarrow \mathcal{I}_{Z}(2d)\rightarrow 0,$$
where $Z \subseteq \P^2$ is a general set of $\binom{d+2}{2}+1$ points.

In \cite{Coskun-Genc-16} Theorem 5.2 and Theorem 5.4, Coskun and Genc proved that there is a unique Ulrich bundle of rank $2k$ for $(\mathds{P}^{2}, \mathcal{O}(2))$. As an application of Theorem \ref{pro-space-Lin-sequence}, we give another elementary proof for the case $k=1$.

\begin{pro}\label{pro-space-Lin-n=d=r=2}
  There is a unique Ulrich bundle of rank $2$ for $(\mathds{P}^{2}, \mathcal{O}(2))$.
\end{pro}
\begin{proof}
  We have found an Ulrich bundle $\E$ in the above example. Let $\E'$ be another Ulrich bundle of rank $2$ for $(\mathds{P}^{2}, \mathcal{O}(2))$, which is determinated by the morphism $\alpha' : \O(-2) \rightarrow \O^{3}(-1)$. In general, we can assume that
  $$\alpha'=(a_{1,1}x+a_{2,1}y+a_{3,1}z\;\;a_{1,2}x+a_{2,2}y+a_{3,2}z\;\;a_{1,3}x+a_{2,3}y+a_{3,3}z).$$
  Since for any point $p \in \P^2$ we have $\mathrm{rank}(\alpha'_{p})=1$, the determinant of matrix $(a_{i,j})$ is non zero. Hence, $A=(a_{i,j}) \in GL(3,k)$ is an automorphism of $\O^3(-1)$ and we have the commutative graph
  $$\xymatrix{
    0 \ar[r] & \O(-2) \ar[r]^{\alpha}\ar[d]_{id}^{\sim}  & \O^{3}(-1)\ar[r]\ar[d]_{A}^{\sim} & \E(-2) \ar[r]\ar[d] & 0 \\
    0 \ar[r] & \O(-2) \ar[r]^{\alpha'}  & \O^{3}(-1)\ar[r]  & \E'(-2)\ar[r] & 0
  }$$
  where $\alpha=(x\;\;y\;\;z)$. So we have $\E' \simeq \E$.
\end{proof}

\begin{eg}
  Consider the case $n=3,r=2$ and $d=2$. By Corollary \ref{pro-space-Lin-sequence-cor3}, finding a rank 2 Ulrich bundle for $(\P^3, \O(2))$ is equivalent to finding two morphisms $\alpha : \O^{1}(-3) \rightarrow \O^{4}(-2)$ and $\beta: \O^{4}(-2) \rightarrow \O^{5}(-1)$, such that $\mathrm{Im}\alpha =\mathrm{ker}\beta$ and the cekernel of $\beta$ is locally free and $H^3(\alpha(-2)): H^3(\O^{1}(-5))\rightarrow H^3(\O^{4}(-4))$ is injective. For example, if we assume that $\{x,y,z,w\}$ are the homogeneous coordinates of $\P^3$, we can take
  \begin{equation*}
    \alpha=\left(
    \begin{array}{cccc}
      x & y & z & w
    \end{array}
    \right),
    \beta=\left(
    \begin{array}{rrrrr}
       y &  z &  w &  0 &  0  \\
      -x &  0 &  0 &  z &  w  \\
       w & -x &  0 & -y &  0  \\
      -z &  0 & -x &  0 & -y
    \end{array}
    \right).
  \end{equation*}
  Since $\alpha \beta=0$ and for any point $p \in \P^3$ the rank of $\beta_{p}$ is 3, we have $\mathrm{Im}\alpha =\mathrm{ker}\beta$ and the cekernel of $\beta$ is locally free. Let $\F$ be the cokernel of $\alpha$, then $\F \simeq \mathcal{T}_{P^3}(-3)$ by definition. Since $h^2(\mathcal{T}_{P^3}(-5))=0$ by Bott formula, we can get that $H^3(\alpha(-2)): H^3(\O^{1}(-5))\rightarrow H^3(\O^{4}(-4))$ is injective. Hence, the cekernel $\G$ of $\beta$ is locally free, and $\G(2)$ is an Ulrich bundle of rank $2$ for $(\P^3, \O(2))$.
\end{eg}

In \cite{Eisenbud-Schreyer-Weyman-03} Proposition 5.11, Eisenbud, Schreyer and Weyman give another example for this case. Let $\F$ be the null correlation bundle on $\P^3$, which is defined by the exact sequence
$$0\rightarrow \F \rightarrow \mathcal{T}_{P^3}(-1)\rightarrow \O(1)\rightarrow 0.$$
Then $\E\doteq \F(2)$ is an rank 2 Ulrich bundle for $(\P^3, \O(2))$. Moreover, they proved the uniqueness of rank 2 Ulrich bundles for $(\P^3, \O(2))$. Similarly, as an application of Theorem \ref{pro-space-Lin-sequence}, we give another proof for this uniqueness.

\begin{pro}
  There is a unique Ulrich bundle of rank $2$ for $(\mathds{P}^{3}, \mathcal{O}(2))$.
\end{pro}
\begin{proof}
  We have found an Ulrich bundle $\E$ in the above example. Let $\E'$ be another Ulrich bundle of rank $2$ for $(\mathds{P}^{3}, \mathcal{O}(2))$, which is determinated by the morphism $\alpha' : \O(-3) \rightarrow \O^{4}(-2)$ and $\beta': \O^{4}(-2) \rightarrow \O^{5}(-1)$. If we assume that $\F'$ is the cokernel of $\alpha'$, as the proof of Proposition \ref{pro-space-Lin-n=d=r=2}, we have $\F' \simeq \mathcal{T}_{P^3}(-3)$. Without loss of generality, we can assume that $\F' = \mathcal{T}_{P^3}(-3)$. Now we have the commutative graph
  $$\xymatrix{
      &    & 0\ar[d] & 0 \ar[d] &  \\
      &    & \O\ar[r]^{\sim}\ar[d] & \O \ar[d]^{s'} &  \\
    0 \ar[r] & \Omega^2(2) \ar[r]\ar[d]^{\sim}  & \O^{6}\ar[r]\ar[d]^{\pi} & \Omega(2) \ar[r]\ar[d] & 0 \\
    0 \ar[r] & \mathcal{T}_{P^3}(-2) \ar[r]  & \O^{5}\ar[r]\ar[d]  & \E'(-1)\ar[r]\ar[d] & 0 \\
      &    & 0 & 0  &
  }$$
  which implies that $\E'$ is determinated by a section $s'$. Since $\E'(-1)$ is locally free, $s' \in H^0(\Omega(2))$ is base point free. If we consider the following short exact sequence
  $$0 \rightarrow \Omega(2)\rightarrow \O^{4}(1) \rightarrow \O(2) \rightarrow 0,$$
  then $H^0(\Omega(2)) \subseteq H^0(\O^{4}(1))$. Let $\{x_{i}\}$ be the homogeneous coordinates of $\P^3$, we can assume that $s'=(s_j')=(\sum_{i=1}^{4}a_{i,j}'x_{i})$ where $(a_{i,j}') \in GL(4, k)$. Similarly, if $\E$ is determinated by a section $s$, we also have $s=(s_j)=(\sum_{i=1}^{4}a_{i,j}x_{i})$ with $(a_{i,j}) \in GL(4, k)$. Let $B=(a_{i,j}')(a_{i,j})^{-1}$. We can get the commutative graph
  $$\xymatrix{
    0 \ar[r] & \O \ar[r]^{s}\ar[d]_{id}^{\sim}  & \Omega(2)\ar[r]\ar[d]_{B}^{\sim} & \E(-1) \ar[r]\ar[d] & 0 \\
    0 \ar[r] & \O \ar[r]^{s'}  & \Omega(2)\ar[r]  & \E'(-1)\ar[r] & 0
  }$$
  So we have $\E' \simeq \E$.
\end{proof}

\begin{eg}
  Consider the case $n=2,r=3$ and $d=3$. By Corollary \ref{pro-space-Lin-sequence-cor2}, finding a rank 3 Ulrich bundle for $(\P^2, \O(3))$ is equivalent to finding a morphism $\alpha : \O^{3}(-2) \rightarrow \O^{6}(-1)$， such that the cekernel of $\alpha$ is locally free and $H^2(\alpha(-3)): H^2(\O^{3}(-5))\rightarrow H^2(\O^{6}(-4))$ is surjective. By Serre duality, The surjection of $H^2(\alpha(-3))$ is equivalent to the injection of $\delta: \mathrm{Hom}(\O^{6}(-4),\O(-3))\rightarrow \mathrm{Hom}(\O^{3}(-5),\O(-3))$. If we assume that $\{x_1,x_2,x_3\}$ are the homogeneous coordinates of $\P^2$, and
  $$\alpha=(\sum_{p=1}^{3} a_{i,j}^{p}x_{p}) \in \mathrm{Hom}(\O^{3}(-2),\O^{6}(-1)),\;\;f=(\sum_{q=1}^{3}b^{j,q}x_{q})\in \mathrm{Hom}(\O^{6}(-4),\O(-3)).$$
  We have the relation of equivalence
  $$\delta (f)=f \circ \alpha =\alpha f=(\sum_{j=1}^{6}(\sum_{p=1}^{3} a_{i,j}^{p}x_{p})(\sum_{q=1}^{3}b^{j,q}x_{q}))=(0)$$
  \begin{equation}
    \Leftrightarrow
    \begin{cases}
      \sum_{j=1}^{6} a_{i,j}^{p}b^{j,p}=0, & \mbox{for } 1\leq i \leq 3, 1 \leq p \leq 3; \\
      \sum_{j=1}^{6} a_{i,j}^{p}b^{j,q}+\sum_{j=1}^{6} a_{i,j}^{q}b^{j,p}=0, & \mbox{for } 1\leq i \leq 3, 1 \leq p < q \leq 3.
    \end{cases}
  \end{equation}
  Let
  \begin{equation*}
    \begin{cases}
      y_{6(p-1)+j}\doteq b^{j,p}, & 1\leq j \leq 6, 1 \leq p \leq 3 \\
      A_{p}\doteq (a_{i,j}^{p})_{1\leq i \leq 3, 1\leq j \leq 3}, &  1 \leq p \leq 3 \\
      B_{p}\doteq (a_{i,j}^{p})_{1\leq i \leq 3, 4\leq j \leq 6}, &  1 \leq p \leq 3
    \end{cases}
  \end{equation*}
  Then the $18 \times 18$ matrix of coefficients of the equations (3) of $\{y_{k}\}$ is
  \begin{equation*}
    C\doteq \left(
    \begin{array}{cccccc}
      A_1 & B_1 & 0   & 0   & 0   & 0   \\
      0   & 0   & A_2 & B_2 & 0   & 0   \\
      0   & 0   & 0   & 0   & A_3 & B_3 \\
      A_2 & B_2 & A_1 & B_1 & 0   & 0   \\
      0   & 0   & A_3 & B_3 & A_2 & B_2 \\
      A_3 & B_3 & 0   & 0   & A_1 & B_1
    \end{array}
    \right).
  \end{equation*}
  So $\delta$ is injective if and only if $\mathrm{det}(C)\neq 0$. Hence, in order to find an Ulrich bundle of rank 2 for $(\P^2, \O(3))$, it's enough to find 6 matrixes $\{A_i,B_i\} \subseteq M_{3\times 3}(k)$ such that $\mathrm{det}(C)\neq 0$. For example, if we assume that
  \begin{equation*}
    \alpha=\left(
    \begin{array}{cccccc}
      x_1 & x_2 & x_2+x_3 & x_2+x_3 & x_3 & x_1+x_2 \\
      x_3 & x_1+x_3 & x_2 & x_1+x_2+x_3 & 0 & x_2 \\
      x_1 & x_2+x_3 & x_1+x_2+x_3 & x_2+x_3 & x_3 & x_3
    \end{array}
    \right),
  \end{equation*}
  then $\mathrm{det}(C)=4\neq 0$. Let $\G$ be the cekernel of $\alpha$, we have that $\G(3)$ is an Ulrich bundle of rank $3$ for $(\P^2, \O(3))$.
\end{eg}

In \cite{Coskun-Genc-16}, Coskun and Genc conjectured that for all odd integer $d$ there exists a rank 3 Ulrich bundle for $(\P^2,\O(d))$. With the aid of computer technology, people know that it is correct when $d$ is very small, but no concrete example is given. Here, we give the first example for the case $d=3$.

\bigskip
\noindent \textbf{Acknowledgement.}\;\;I am heavily indebted to my supervisor Professor Jinxing Cai for his patient guidance and helping me out of the difficulties during the hard time of this paper. I also thank Zhaonan Li, Dong Zhang and Xufeng Guo for their assistance of computer technology.

\bigskip

\bigskip

Zhiming Lin


\bigskip

\begin{thebibliography}{99}

\bibitem{Aprodu-Farkas-Ortega-12} M. Aprodu, G. Farkas, A. Ortega: Minimal resolutions, Chow forms and Ulrich bundles on K3 surfaces, arXiv:1212.6248.

\bibitem{Aprodu-Costa-Miro Roig-16} M. Aprodu, L. Costa, R.M. Miro Roig: Ulrich Bundles on Ruled Surfaces, arXiv:1609.08340.

\bibitem{Beauville-15.12} A. Beauville: Ulrich bundles on abelian surfaces. Proc. Amer. Math. Soc. 144 (2016), no. 11, 4609-4611.

\bibitem{Beauville-16.07} A. Beauville: Ulrich bundles on surfaces with $p_{g} = q = 0$, arXiv:1607.00895.

\bibitem{Beauville-16.10} A. Beauville: An Introduction to Ulrich Bundles, arXiv:1610.02771.

\bibitem{Casnati-16.09} G. Casnati: Special Ulrich bundles on non-special surfaces with $p_{g}=q=0$, arXiv:1609.07915.

\bibitem{Coskun-Genc-16} E. Coskun, O. Genc: Ulrich bundles on Veronese Surfaces, arXiv:1609.07130.

\bibitem{Eisenbud-Schreyer-Weyman-03} D. Eisenbud, F.O. Schreyer, J. Weyman: Resultants and Chow forms via exterior syzigies, Amer. Math. Soc. 16 (2003), 537-579.

\bibitem{Herzog-Ulrich-Backelin-91} J. Herzog, B. Ulrich, J. Backelin: Linear maximal Cohen-Macaulay modules over strict complete intersections. J. Pure Appl. Algebra 71 (1991), no. 2-3, 187-202.

\bibitem{Hirzebruch-66}F. Hirzebruch: Topological methods in algebraic geometry, Springer, 1966.

\bibitem{Okonek-Schneider-Spindler-88} C. Okonek, M. Schneider, H. Spindler: Vector Bundles on Complex Projective Spaces, Birkhauser Verlag, 1988.

\bibitem{Ulrich-84} B. Ulrich: Gorenstein rings and modules with high numbers of generators. Math. Z. 188 (1984), no. 1, 23–32.

\end{thebibliography}
\end{document}